\documentclass[a4paper]{amsart}

\usepackage[utf8]{inputenc}
\usepackage{mathtools}
\usepackage{amssymb,amsfonts,amsmath}
\usepackage{enumerate}
\usepackage{graphicx}
\usepackage{verbatim}
\usepackage{tikz-cd}

\numberwithin{equation}{section}
\newtheorem{theorem}{Theorem}[section]
\newtheorem{corollary}[theorem]{Corollary}
\newtheorem{lemma}[theorem]{Lemma}
\newtheorem{additio}[theorem]{Additio}
\newtheorem{proposition}[theorem]{Proposition}

\theoremstyle{definition}
\newtheorem{definition}[theorem]{Definition}

\theoremstyle{remark}
\newtheorem{remark}[theorem]{Remark}

\newcommand{\F}{\mathbb{F}}
\newcommand{\N}{\mathbb{N}}

\newcommand{\R}{\mathbb{R}}
\newcommand{\Z}{\mathbb{Z}}

\DeclareMathOperator{\Alt}{Alt}
\DeclareMathOperator{\Aut}{Aut}

\DeclareMathOperator{\id}{id}

\DeclareMathOperator{\ord}{ord}

\DeclareMathOperator{\RiSt}{RiSt}

\DeclareMathOperator{\SO}{SO}

\DeclareMathOperator{\len}{len}
\DeclareMathOperator{\St}{St}

\DeclareMathOperator{\Sym}{Sym}

\newcommand{\abs}[1]{\vert #1 \vert}

\newcommand{\T}{\mathcal{T}}

\newcommand{\defeq}{\mathrel{\mathop{:}}=}

\renewcommand{\epsilon}{\varepsilon}

\title[Residually finite groups as subgroups of branch groups]{Realizing residually finite groups as subgroups of branch groups}
\author[S. Kionke]{Steffen Kionke}
\author[E. Schesler]{Eduard Schesler}
\address{FernUniversit\"at in Hagen \\ Fakult\"at f\"ur Mathematik und Informatik \\
58084 Hagen}
\email{steffen.kionke@fernuni-hagen.de}
\email{eduard.schesler@fernuni-hagen.de}
\thanks{Funded by the Deutsche Forschungsgemeinschaft (DFG, German Research Foundation) - 441848266}

\begin{document}
\begin{abstract}
We prove that every finitely generated, residually finite group $G$ embeds into a finitely generated perfect branch group $\Gamma$ such that many properties of $G$ are preserved under this embedding.
Among those are the properties of being torsion, being amenable, and not containing a non-abelian free group.
As an application we construct a finitely generated, non-amenable torsion branch group.
\end{abstract}
\maketitle

\section{Introduction}

The history of amenable groups goes back to 1914 when Hausdorff~\cite[Appendix]{Hausdorff14} discovered that the unit sphere $S \subseteq \R^3$ has a countable subset $F$ whose complement $S' = S \backslash F$ admits a paradoxical decomposition.
More precisely, he showed that $S'$ decomposes as a union of $3$ disjoint subsets $A, B, C$ such that $A,B,C$, and $B \cup C$ are congruent.
This result, known as the Hausdorff paradox and precursor of the Banach-Tarksi paradox~\cite{BanachTarski24}, clearly implies that there is no finitely additive probability measure on $S$ that is defined on all of its subsets and assigns the same measure to any two congruent subsets.
The notion of amenability arose from von Neumann's~\cite{vNeumann29} insight that the source of these paradoxes can be traced back to the existence of non-abelian free subgroups in $\SO(3)$.
Regarding the general setting in which a group $G$ acts freely on a set $X$, he showed that the existence of non-abelian free subgroups in $G$ gives rise to a paradoxical decomposition of $X$.
This insight led von Neumann to initiate the study of amenable\footnote{Von Neumann designated such group with the German word ``messbar'' (``measurable'' in English).
The fact that such groups are now called amenable goes back to Day~\cite{Day49}, apparently as a pun.} groups, i.e.\ groups that admit a finitely additive, left invariant, probability measure defined on all of its subsets.
It follows that the class $\mathcal{AG}$ of amenable groups is contained in the class $\mathcal{NF}$ of groups that do not contain non-abelian free subgroups.
The question whether the classes $\mathcal{AG}$ and $\mathcal{NF}$ coincide is known as the von Neumann problem and has its first written appearance in work of Day~\cite{Day57}.
In the same article Day introduced a rich source of amenable groups by defining the class $\mathcal{EG}$ of elementary amenable groups as the smallest one that contains all finite and abelian groups, and that is closed under taking subgroups, extensions, quotients and direct limits.
Since all groups in $\mathcal{NF}$ known at that time were elementary amenable, Day extended von Neumann's problem by asking whether the inclusion $\mathcal{EG} \subseteq \mathcal{AG}$ is proper.
We refer to this question as Day's problem.
Meanwhile von Neumann's problem and Day's problem were shown to have negative solutions.
In the case of von Neumann's problem, the first counterexample goes back to Ol’shanskii~\cite{Olshanskii80}, who constructed finitely generated, non-amenable groups of bounded exponent.
Note that the latter condition clearly prevents the existence of non-abelian free subgroups.
Another source of counterexample was provided by Ershov~\cite{Ershov08}.
Like Ol’shanskii's examples, the groups constructed by Ershov have their origin in the solution of the Burnside problem.
In fact, Ershov's construction is based on Golod-Shafarevich groups, which are the first known finitely generated, infinite torsion groups.
However, in contrast to Ol’shanskii's examples, which can be chosen to be simple, Ershov's groups are residually finite.
This is essential for us as it allows us to embed Ershov's groups into torsion branch groups; see Subsection~\ref{subsec:branch-groups}.

Branch groups have their origin in the discovery of Grigorchuk's groups~\cite{Grigorchuk80}, yet another completely different type of finitely generated, infinite torsion groups, which provided the first counterexample to Day's problem.
More precisely, Grigorchuk showed that his groups have intermediate word growth, which was shown by Chou~\cite{Chou80} to be impossible for finitely generated elementary amenable groups.
In fact, it turned out that counterexamples to Day's problem are ubiquitous in the class of branch groups, see e.g. the examples discussed in~\cite{JuschenkoNekrashevychdelaSalle16}.
However, there are also finitely generated branch groups known that are non-amenable, see e.g.~\cite{SidkiWilson03} and~\cite{KionkeSchesler21} for examples of branch groups that contain non-abelian free groups.
It is a natural question, raised by Bartholdi, Grigorchuk, and \v{S}uni\'{k}, whether the class of branch groups also contains counterexamples to von Neumann's problem, see e.g.~\cite[Question 20]{BGS-branch},~\cite[Problem 9.1. a)]{Grigorchuk05}, and~\cite[15.14. b)]{KhukhroMazurov14}.
By applying the following embedding theorem on Ershov's groups discussed above, we will see that this is indeed the case.

\begin{theorem}\label{thm:intro-main}
Every finitely generated, residually finite group $G$ embeds into a finitely generated perfect branch group $\Gamma$ such that
\begin{enumerate}
\item\label{it-torsion} if $G$ is torsion, then $\Gamma$ is torsion,
\item\label{it-amenable} if $G$ is amenable, then $\Gamma$ is amenable,
\item\label{it-free} if $G$ does not contain a non-abelian free group, then $\Gamma$ does not contain a non-abelian free group,
\item\label{it-law} if $G$ satisfies the law $w$ up to powers, then $\Gamma$ satisfies $w$ up to powers.
\end{enumerate}
\end{theorem}
A group $\Gamma$ is said to satisfy a law $w$ \emph{up to powers}, if for all $\gamma_1,\dots, \gamma_r$, there are exponents $t_1,\dots,t_r \in \N$ such that $\langle\gamma_1^{t_1},\dots,\gamma_r^{t_r}\rangle$ satisfies $w$.
As an example, we note that every law $w$ that is satisfied in a finite index subgroup, is a law up to powers.

\begin{corollary}\label{cor:intro-main}
There is a finitely generated, non-amenable torsion branch group $\Gamma$.
In particular, the von Neumann problem has a negative solution in the class of finitely generated branch groups.
\end{corollary}
\begin{proof}
It was proven by Ershov~\cite{Ershov08} that there exists a finitely generated, residually finite, non-amenable torsion group $G$.
We can apply Theorem~\ref{thm:intro-main} to embed $G$ into a finitely generated, torsion branch group $\Gamma$.
Now the corollary follows from the easy fact that groups containing non-amenable groups are non-amenable.
\end{proof}

Recall that an action of a group $G$ on a set $X$ is called amenable if $X$ admits a $G$-invariant probability measure defined on all of its subsets.
It was asked by Greenleaf~\cite{Greenleaf69} whether the existence of an amenable action of a group $G$ implies that $G$ is amenable.
Of course, to avoid easy counterexamples, one should add the assumption that the action is faithful and transitive.
But even in this formulation the question was answered in the negative by von Douwen~\cite{vanDouwen90}, who constructed such an action for the non-abelian free group of rank $2$.
Since then, the study of amenable actions of non-amenable groups gained quite some interest and many other examples were found, see e.g.~\cite{MonodPopa03},~\cite{GlasnerMonod07}, and~\cite{GrigorchukNekrashevych07}.
To the best of our knowledge, all non-amenable groups admitting such an action that appear in the literature so far contain non-abelian free groups.
The following Corollary, which will be deduced by applying a result of Grigorchuk and Nekrashevych~\cite{GrigorchukNekrashevych07} to the group $\Gamma$ in Corollary~\ref{cor:intro-main}, tells us that this is not necessarily the case.

\begin{corollary}\label{cor:amenable-action}
The group $\Gamma$ in Corollary~\ref{cor:intro-main} admits a continuous action on a Cantor set $\mathfrak{C}$ such that the restriction of this action to each orbit is amenable and faithful.
\end{corollary}

Let $\Gamma$ be a branch group as in Corollary~\ref{cor:intro-main}.
Using the fact that every proper quotient of a branch group is virtually abelian~\cite[Proposition 6]{DelzantGrigorchuk08}, it follows that $\Gamma$ is \emph{just-infinite}, i.e.\ an infinite group all of its proper quotients are finite.
In particular, every proper quotient of $\Gamma$ is amenable, which makes $\Gamma$ a minimal counterexample to von Neumann's problem.
Interestingly, the same reasoning applies for variations of the other two types of torsion groups discussed above.
In fact, the examples constructed by Ol’shanskii are Tarski monsters, which makes them simple and hence just-infinite.
A just-infinite variation of Golod-Shafarevich groups was constructed by Ershov and Jaikin-Zapirain~\cite{ErshovJaikin10}.
The groups they constructed are residually finite torsion groups with property $(\mathrm{T})$, from which they can be easily seen to provide counterexamples to von Neumann's problem.
Moreover, the groups of Ershov and Jaikin-Zapirain are hereditarily just-infinite, which means that they are residually finite and all of their finite index subgroups are just-infinite.
Recall from Wilson's celebrated classification of just-infinite groups that every just-infinite group is either a branch group or virtually a direct product of finitely many isomorphic copies of a group that is either simple or hereditarily just-infinite.
In view of this classification and the just mentioned results, we see that Corollary~\ref{cor:intro-main} provides us with the last piece needed to get the following.

\begin{corollary}\label{cor:just-infinite}
Each of the $3$ classes of Wilson's classification of finitely generated, just-infinite groups contains a counterexample to von Neumann's problem.
\end{corollary}

The branch group $\Gamma$ in our main theorem is far from being unique. Our proof is based on constructing an infinite family $\Gamma^{\alpha,\beta}$ of branch groups, parametrized by $\alpha,\beta$ in a compact metric space. We then use Baire's category theorem to show that for a dense set of parameters $\Gamma^{\alpha,\beta}$ has the desired property.
Our construction provides additional flexibility. In particular, it is possible to freely choose the profinite topology induced on $G$ by the finite index subgroups of $\Gamma$.
\begin{additio}\label{additio:profinite}
In the setting of Theorem \ref{thm:intro-main}.
Let $H$ be a profinite group that contains $G$ as a dense subgroup. Then the branch group $\Gamma$ can be chosen so that the image of the induced map $\widehat{G} \to \widehat{\Gamma}$ of profinite completions is isomorphic to $H$.
\end{additio}
Given a finitely generated profinite group it can be difficult to determine whether or not it is the profinite completion of an abstract finitely generated group. For instance, $\prod_{p} \mathrm{SL}_2(\F_p)$ cannot be the profinite completion of a finitely generated group (see \cite{KassabovNikolov06}). In this regard, we think that it is somewhat surprising, that our construction shows that every finitely generated profinite group $H$ is isomorphic to the image of a homomorphism $\widehat{G}_1 \to \widehat{G}_2$ induced by an embedding of finitely generated groups.

\section{Groups acting on rooted trees}

Let us start by recalling some standard terminology in the area of groups acting on rooted trees.

\subsection{Branch groups}\label{subsec:branch-groups}

Let $X = (X_n)_{n \in \N}$ be a sequence of finite sets of cardinality $\abs{X_n} \geq 2$, which we will think of as alphabets.
For each $\ell \in \N_0$, we consider the set of words of length $\ell$ given by $X^{\ell} = X_1 \times \ldots \times X_{\ell}$, where $X^{0} = \{\emptyset\}$ is defined to be the singleton consisting of the unique word $\emptyset$ of length $0$.
The \emph{spherically homogeneous rooted tree associated to $X$}, denoted by $\T = \T_X$, is the rooted tree with vertex set $X^{\ast} = \bigcup \limits_{\ell=0}^{\infty} X^{\ell}$ and root vertex $\emptyset$, where two vertices $v,w$ are connected by an edge if and only if there is a letter $x \in X_n$ for some $n \in \N$ such that either $v = wx$ or $w = vx$.
The groups we are interested in are subgroups of the group $\Aut(\T)$ of all automorphisms of $\T$ that fix the root $\emptyset$.
Note that the distance of a vertex $v$ to the root, which we call the \emph{level} of $v$, is preserved under the action of $\Aut(\T)$ and coincides with the word length of $v$.
It therefore follows that the sets $X^{\ell}$ of vertices of level $\ell$ are stable under the action of $\Aut(\T)$, which provides us with a canonical homomorphism $\pi_{\ell} \colon \Aut(\T) \rightarrow \Sym(X^{\ell})$.
Consider now a subgroup $G$ in $\Aut(\T)$.
We say that $G$ acts \emph{spherically transitively} on $\T$ if the action of $G$ on $X^{\ell}$ via $\pi_{\ell}$ is transitive for every $\ell \in \N_0$.
For each such $\ell$, the \emph{level $\ell$ stabilizer subgroup} in $G$ is defined by
\[
\St_G(\ell) \defeq \bigcap \limits_{v \in X^{\ell}} \St_G(v),
\]
where $\St_G(v)$ denotes the stabilizer of $v$ in $G$.
Note that $\St_G(\ell)$ coincides with the kernel of the restriction of $\pi_{\ell}$ to $G$, which makes it a normal subgroup of finite index in $G$.
The \emph{rigid stabilizer} of $v$ in $G$, denoted by $\RiSt_G(v)$, is the subgroup of $\St_G(v)$ consisting of those elements that fix every word that does not contain $v$ as an initial subword.
The subgroup of $\St_G(\ell)$ that is generated by the groups $\RiSt_G(v)$ with $v \in X^{\ell}$ is called the \emph{rigid level $\ell$ stabilizer subgroup} in $G$ and will be denoted by $\RiSt_G(\ell)$.
It can be easily seen that $\RiSt_G(\ell)$ is a normal subgroup of $G$.
If moreover $\RiSt_G(\ell)$ has finite index in $G$ for every $\ell \in \N$ and $G$ acts spherically transitively on $\T$, then $G$ is said to be a \emph{branch subgroup} of $\Aut(\T)$.
More generally, we say that a group $G$ is a \emph{branch group} if it is isomorphic to a branch subgroup of the automorphism group of a spherically homogeneous rooted tree.

For each $j \in \N$, let $\T_{j}$ be the spherically homogeneous rooted tree associated to the subsequence $(X_j, X_{j+1}, \ldots)$ of $X$.
Consider an element $h \in \Aut(\T_j)$.
For every vertex $u$ in $\T_j$ of level $\ell$, there is a unique automorphism $h|_u \in \Aut(\T_{j+\ell})$ that satisfies
\begin{equation}\label{eq:decomposition-rooted-and-section}
h(uv) = h(u) h|_u(v);
\end{equation}
$h|_u$ is called the \emph{section} of $h$ at $u$.
We note that $(gh)|_u = g|_{h(u)} h|_u$ for all $g,h \in \Aut(\T_j)$ and $\ ^h g|x = h|_{gh^{-1}x} g|_{h^{-1}x}h^{-1}|_x$.
In view of~\eqref{eq:decomposition-rooted-and-section} we see that every automorphism of $\T_j$ is determined by $(h|_x)_{x \in X_j}$ and its action on $X_j$.

\subsection{Stabilized sections and shrinking sequences}\label{sec:stabilized}
Let us fix two groups $Q,G$ and
 a sequence $X = (X_n)_{n \in \N}$ of finite sets of cardinality $\abs{X_n} \geq 2$ each equipped with actions of $Q$ and $G$.
Let $Q_n, G_n \subseteq \Sym(X_n)$ denote the images of $Q$ resp. $G$ in the symmetric group and let $A_n = \langle Q_n, G_n \rangle$. Throughout we will make the following assumptions for all $n \in \N$
\begin{enumerate}[({A}1)]
\item $G$ and $Q$ are finitely generated,
\item $Q$ is perfect,
\item\label{it:generation} $A_n$ is transitive and generated by the $G_n$-conjugates of $Q_n$.
\end{enumerate}
Under these assumptions $A_n$ is perfect, since it is generated by perfect groups.

We consider a number of tree automorphisms.
First, we define an action of $A_j$ on $\T_j$ by \emph{rooted} automorphisms.
Here $A_j$ permutes the subtrees hanging below the vertices in $X_j$ according to its action on $X_j$.
In other words, we define $\rho_j\colon A_j \to \Aut(\T_j)$ by setting $\rho_j(a)x_jx_{j+1}\cdots x_k = (ax_j)x_{j+1}\cdots x_k$.
Since $\rho_j$ is injective, we will identify $A_j$ with its image under $\rho_j$.
For $q \in Q$ (resp. $g \in G$) we write $q_j$ (resp. $g_j$) to denote the image of $q$ (resp.\ $g$) in $A_j$.

In addition, we define actions of $Q$ and $G$ via \emph{directed} automorphisms on $\T_j$.
In each $X_i$ we fix a designated point $o \in X_i$, for simplicity we will not distinguish them in our notation. Define $\mathcal{S} = \prod_{i=1}^\infty X_i\setminus\{o\}$. 
For every $\alpha \in \mathcal{S}$ and every $q \in Q$ we recursively define a family of automorphisms $\tilde{q}_{[j]}^\alpha \in \Aut(\T_j)$ as follows.
Each $\tilde{q}_{[j]}^\alpha$  acts trivially on the first level of $\T_j$ and its section at $x \in X_j$ is given by
\[
	\tilde{q}_{[j]}^\alpha|_x = \begin{cases}
	 \tilde{q}_{[j+1]}^{\alpha} & \text{ if } x = o\\
	 q_{j+1} & \text{ if } x = \alpha_j\\
	 \id_{\T_{j+1}} & \text{ if } x \not\in \{o,\alpha_j\}.
	\end{cases}\]
In the same way every element $\beta \in \mathcal{S}$ and every $g \in G$ gives rise to an automorphism $\tilde{g}_{[j]}^\beta \in \Aut(\T_j)$.
We observe that $g \mapsto \tilde{g}^\beta_{[j]}$ defines a homomorphism from $G$ to $\Aut(\T_j)$ whose kernel is given by $\bigcap_{k \geq j} \ker(G \to G_k)$. The same applies to the homomorphisms $q \mapsto q^\alpha_{[j]}$.
Here we will study the groups
 \[ 
 	\Gamma^{\alpha,\beta}_{j} \defeq \langle A_j, \tilde{Q}^\alpha_{[j]}, \tilde{G}_{[j]}^\beta \rangle.
\]
Recall that the \emph{permutational wreath product} $K \wr_{Y} H$ of a group $K$ and a group $H$ acting on a set $Y$ is defined as the semidirect product $K^Y \rtimes H$, where $H$ acts on $K^{Y}$ by permuting the coordinates.

\begin{lemma}\label{lem:branch-group-arg}
Let $\alpha, \beta \in \mathcal{S}$. In addition to our standing assumptions we assume that for all $j \in \N$ the points $o$ and $\alpha_j$ have distinct stabilizers in $A_j$.
Then $\Gamma_j^{\alpha,\beta}$ is a finitely generated, perfect branch subgroup of $\Aut(\T_j)$ and $\Gamma_j^{\alpha,\beta} \cong \Gamma_{j+1}^{\alpha,\beta} \wr_{X_j} A_j$ for all $j \in \N$.
\end{lemma}
\begin{proof}
We note that $\Gamma_j^{\alpha,\beta}$ is finitely generated, since it is generated by three finitely generated subgroups.

We consider the action of $\Gamma_j^{\alpha,\beta}$ on $X_j$ via $A_j$.
To prove that $\Gamma_j^{\alpha,\beta}$ is a branch subgroup of $\Aut(\T_j)$, we show that $\RiSt_{\Gamma_j^{\alpha,\beta}}(\ell)$ coincides with $\St_{\Gamma_j^{\alpha,\beta}}(\ell)$ for every $\ell \in \N$.
By induction, it suffices to verify that $\RiSt_{\Gamma_j^{\alpha,\beta}}(1) = \St_{\Gamma_j^{\alpha,\beta}}(1)$ and $\Gamma_j^\beta \cong \Gamma_{j+1}^\beta \wr_{X_j} A_j$.
Since all first level sections of $\Gamma_{j}^{\alpha,\beta}$ lie in $\Gamma_{j+1}^{\alpha,\beta}$, it is enough to show that $\prod_{x \in X_j}\Gamma_{j+1}^{\alpha,\beta}$ is contained in $\RiSt_{\Gamma_j^{\alpha,\beta}}(1)$.
To see this, we use a commutator trick that was extracted by Segal
\cite[Lemma 4]{Segal01} from a proof of Grigorchuk~\cite[Theorem 4]{Grigorchuk00}.
By assumption $o$ and $\alpha_j$ have distinct stabilizers. Since $A_j$ is finite and acts transitively on $X_j$, none of the stabilizers contains the other.  This means, there is an element $a \in \St_{A_j}(o)$ such that $a.\alpha_j \neq \alpha_j$. Then for all $p,q \in Q$ we have
\[
[\ ^{a}(\tilde{p}^\alpha_{[j]}),\tilde{q}^\alpha_{[j]}]  = [\tilde{p}^\alpha_{[j+1]},\tilde{q}^\alpha_{[j+1]}].
\]
Since $Q$ is perfect, we deduce that $\tilde{Q}_{[j+1]}^\alpha$ lies in $\RiSt(o)$.
By taking products $\tilde{q}^\alpha_{[j]}(\tilde{q}^\alpha_{[j+1]})^{-1}$, we see that $\RiSt(\alpha_j)$ contains $Q_{j+1}$. Since $A_j$ acts transitively on the first level of $\T_j$, we conclude that $\RiSt(\beta_j)$ contains $Q_{j+1}$ and conjugation with $\tilde{G}_{[j]}^\beta$ shows that it contains all $G_{j+1}$-conjugates of $Q_{j+1}$.
It therefore follows from assumption (A\ref{it:generation}) that $\RiSt(\beta_j)$ contains $A_{j+1}$ and the transitivity of $A_j$ shows that the rigid stabilizers of first level vertices contain $A_{j+1}$ and in particular $G_{j+1}$.
Let $g \in G$ be arbitrary and
let $k \in \RiSt(\beta_j)$ be such that $k|_{\beta_j} = g_j$. Then $k^{-1}\tilde{g}^\beta_{[j]} = \tilde{g}^\beta_{[j+1]}$ and we conclude that that $\RiSt(o)$ contains $\tilde{G}^\beta_{[j+1]}$.
Using again the transitivity of $A_j$ on the first level of $\T_j$, we conclude that $\Gamma_{j+1}^{\alpha,\beta}$ is contained in, and hence coincides with, the rigid stabilizer of every vertex in the first level of $\T_j$.
This gives us
\[
\Gamma_j^{\alpha,\beta}
= \St_{\Gamma_j^{\alpha,\beta}}(1) \rtimes A_j
\cong (\Gamma_{j+1}^{\alpha,\beta})^{X_j} \rtimes A_j
= \Gamma_{j+1}^{\alpha,\beta} \wr_{X_j} A_j.
\]
Since $A_j$ is perfect, this equation implies that $\Gamma^{\alpha,\beta}_j$ is perfect, which proves the lemma.
\end{proof}

We want to prove that if the sets $X_i$ grow suitably in size and $\alpha$, $\beta$ are generic, then $\Gamma_j^{\alpha,\beta}$ will resemble $Q \times G$ up to powers.
To this end, we use the notion of \emph{stabilized sections} introduced by Petschick in \cite{PetschickPeriodicity}.
Given $g \in \Aut(\T_j)$ and a vertex $u$ of $\T_j$, let $\ell_u(g)$ denote the length of the orbit of $u$ under $g$.
The stabilized section $g\Vert_u$ of $g$ at $u$ is $g\Vert_u \defeq g^{\ell_u(g)}|_u$.
Note that we always have $g\Vert_u \Vert_v = g\Vert_{uv}$.

We assume from now on that $\alpha_j \neq \beta_j$ for all $j \in \N$.
In this situation $\tilde{Q}^\alpha_{[j]}$ and $\tilde{G}^\beta_{[j]}$ commute and the group $\Gamma_j^{\alpha,\beta}$ is a quotient of the free product $F_j = A_j * (Q\times G)$ via a homomorphism $f^{\alpha,\beta}_j \colon F_j \to \Gamma_j^{\alpha,\beta}$ by mapping $a \in A_j$  to itself, $q \in Q$ to $\tilde{q}^\alpha_{[j]}$ and $g \in G$ to $\tilde{g}_{[j]}^\beta$. Every element $w \in F_j$
can be written uniquely in normal form with letters alternating between non-trivial elements in $A_j$ and non-trivial elements in $B = Q \times G$; see \cite[6.2.4]{Robinson}.
We write $\len_B(w)$ for the number of letters from $B$ and we write $\len_A(w)$ for the number of letters in $A_j$.
Note that $|\len_B(w)-\len_A(w)| \leq 1$.
The pair $\len(w) \defeq (\len_B(w), \len_A(w))$ is called the \emph{length} of $w$.
We endow $\Z \times \Z$ with the lexicographical ordering, i.e.\ we have $(b,a) > (b',a')$ if either $b > b'$ or $b=b'$ and $a > a'$.
For $\gamma \in \Gamma_j^{\alpha,\beta}$ we write $\len(\gamma) = (\len_G(\gamma),\len_A(\gamma))$ to denote the minimal length of an element $w \in F_j$ with $f_j^{\alpha,\beta}(w) = \gamma$.
In this case we say that $w$ is a minimal representative of $\gamma$.
In what follows it will be useful to observe that for all $\gamma_1,\gamma_2 \in \Gamma^{\alpha,\beta}_j$ we have
\begin{equation}\label{eq:gamma1-gamma2}
\max(\len(\gamma_1),\len(\gamma_2)) - \min(\len(\gamma_1),\len(\gamma_2)) \leq \len(\gamma_1 \gamma_2) \leq \len(\gamma_1) + \len(\gamma_2),
\end{equation}
which is a direct consequence of the definitions involved.
The next result is a reformulation of~\cite[Lemma 2.4]{PetschickPeriodicity}.

\begin{lemma}\label{lem:dec-len}
For every $\gamma \in \Gamma_j^{\alpha,\beta}$ we have
\begin{enumerate}
\item\label{it:sum-bound} $\sum \limits_{x \in X_j} \len_B(\gamma|_x) \leq \len_B(\gamma)$, and
\item every vertex $x$ of $\T_j$ satisfies $\len_B(\gamma\Vert_x) \leq \len_B(\gamma)$.
\end{enumerate}
\end{lemma}
\begin{proof}
Every $B$-letter in $w$ evaluates to an element which contributes exactly one element in $\tilde{B}_{j+1} \defeq \tilde{Q}^\alpha_{[j+1]}\tilde{G}_{[j+1]}^\beta$ (i.e., one $B$-letter) to the sections on the first level of $\T_j$, hence \eqref{it:sum-bound}.
For the second assertion, the equation $\gamma\Vert_u \Vert_v = \gamma\Vert_{uv}$ allows us to assume that $x$ is a vertex in the first level of $\T_j$.
Consider the orbit length $\ell = \ell_x(\gamma)$ of $x$ under $\gamma$.
Then we have
\begin{equation}\label{eq:gamma-ell-x}
\gamma^\ell |_x = \gamma|_{\gamma^{\ell-1}x} \cdot \gamma|_{\gamma^{\ell-2}x} \cdots \gamma|_x,
\end{equation}
where the vertices $x, \gamma x, \dots, \gamma^{\ell-1}x$ are pairwise distinct.
Then the assertion follows from~\eqref{it:sum-bound}.
\end{proof}

\begin{definition}\label{def:shrinking}
We say that $(\alpha,\beta) \in \mathcal{S}^2$ is \emph{shrinking} if $\alpha_j \neq \beta_j$ for all $j \in \N$ and for each $\gamma \in \Gamma_1^{\alpha,\beta}$ there is a $k \in \N$ such that for all vertices $x$ in $\T_1$ of level $k$ we have $\len(\gamma\Vert_x) \leq (1,0)$; in other words, $\gamma\Vert_x$ lies in $\tilde{B}^{\alpha,\beta}_{k+1}\defeq \tilde{Q}_{[k+1]}^\alpha\tilde{G}_{[k+1]}^\beta$ or in $A_{k+1}$.
\end{definition}

This property will be important for us due to the following simple observation.

\begin{lemma}\label{lem:use-shrink}
Assume that $Q$ is finite.
Let $(\alpha,\beta) \in \mathcal{S}^2$ be shrinking and let $\gamma \in \Gamma_1^{\alpha,\beta}$.
Then for all sufficiently large $k \in \N$, there is some $m \in \N$ such that $\gamma^m \in \St_{\Gamma_1^{\alpha,\beta}}(k)$ and $\gamma^m|_v \in \tilde{G}^{\beta}_{[k+1]}$ for all vertices $v$ of level $k$ in $\T_1$.
\end{lemma}
\begin{proof}
Since $(\alpha,\beta)$ is shrinking, there is a $k_0 \in \N$ such that $\gamma\Vert_v$ lies in $\tilde{B}_{k_0+1}^{\alpha,\beta}$ or $A_{k_0+1}$ for all vertices $v$ of level $k_0$.
Let $k \geq k_0$ be a natural number and let $m$ be a multiple of the orbit sizes of $\gamma$ on level $k$.
Then $\gamma^m$ stabilizes every vertex of level at most $k$.
In particular, $\gamma^m$ stabilizes the vertices on the $k_0$-th level and $\gamma^m|_v = (\gamma\Vert_v)^{m/\ell_v(\gamma)}$ lies in $\tilde{B}_{k_0+1}^{\alpha,\beta}$ or $A_{k_0+1}$ for all $v$ of level $k_0$.
Since sections of rooted elements are trivial, it follows from the definition of $\tilde{B}^{\alpha,\beta}_{k_0+1}$ that the sections $\gamma^m|_v$ are in $A_{k+1}$ or $\tilde{B}^{\alpha,\beta}_{k+1}$ for every vertex $v$ of level $k$.
Let $e$ be a common multiple of the exponent of $Q$ and of $A_{k+1}$.
Then $\gamma^{me} \in \St_{\Gamma_1^{\alpha,\beta}}(k)$ and all sections of $\gamma^{me}$ on level $k$ are contained in $\tilde{G}^\beta_{k+1}$, which proves the lemma.
\end{proof}

We will show that under mild hypothesis most pairs $(\alpha,\beta)$ are shrinking.
To this end we consider non-empty disjoint subsets $Y_i,Y'_i \subseteq X_i \setminus\{o\}$ for every $i \in \N$. these sets will later become useful in ensuring the stabilizer condition in Lemma~\ref{lem:branch-group-arg} and to make sure that $\alpha_j \neq \beta_j$ holds for all $j$. Let $\mathcal{Y}= \prod_{i=1}^\infty Y_i$ and let $\mathcal{Y}'= \prod_{i=1}^\infty Y'_i$.
Let $m_j$ denote the maximal order of the elements in $A_j$.

\begin{proposition}\label{prop:shrinking}
Suppose that for every $n \in \N$ there is some $j$ satisfying 
\[ \abs{Y_j}\abs{Y_j'} > nm_j(\abs{Y_j}+\abs{Y'_j}).\]
Then the set of shrinking pairs is dense in $\mathcal{Y} \times \mathcal{Y}'$.
\end{proposition}

Before we can prove the proposition, we need a technical lemma.
For each $w \in F_j$ we define a set $Z_j(w)$ of pairs $(s,t) \in Y_j\times Y_j'$. If $\len_B(w) > 1$, then $Z_j(w)$ contains those pairs that satisfy $\len_B(f_j^{\alpha,\beta}(w)\Vert_x) = \len_B(w)$ for some vertex $x \in X_j$ and some $(\alpha,\beta) \in \mathcal{Y} \times \mathcal{Y}'$ with $\alpha_j = s$ and $\beta_j = t$.
If $\len_B(w) = 1$, then $Z_j(w)$ consists of all pairs satisfying $\len(f_j^{\alpha,\beta}(w)\Vert_x) > (1,0)$.

\begin{lemma}\label{lem:shrink-size}
Let $w \in F_j$ with $\len(w) > (1,0)$. 
Then 
\[|Z_j(w)| \leq  \len_B(w)m_j(|Y_j|+|Y'_j|).\]
\end{lemma}
\begin{proof}
Let $\gamma = f_j^{\alpha,\beta}(w)$ and let $n = \len_B(w)$.
Then $w$ can be written in the form
$w= \ ^{a_1} b_1 \cdots \ ^{a_n} b_n a$
with $a, a_1,\dots,a_n \in A_j$ and non-trivial $b_1,\dots, b_n \in B$. 
Let $x \in X_j$. Note that $\ell \defeq \ell_x(a) = \ell_x(\gamma)$ since $\gamma$ acts as $a$ on the first level of $\T_j$.
We write $\tilde{b}_i = f^{\alpha,\beta}_j(b_i)$.
A quick calculation reveals
\begin{align*}
& \gamma\Vert_x = \gamma^\ell|_x = \bigl(\ ^{a_1} \tilde{b}_1 \cdots \ ^{a_n} \tilde{b}_n \ ^{aa_1} \tilde{b}_1 \cdots \ ^{aa_n} \tilde{b}_n \cdots \ ^{a^{\ell-1}a_1} \tilde{b}_1 \cdots \ ^{a^{\ell-1}a_n} \tilde{b}_n a^\ell\bigr)|_x \\
=\ &\tilde{b}_1|_{a_1^{-1}x} \cdots \ \tilde{b}_n|_{a_n^{-1}x} \  \tilde{b}_1|_{a_1^{-1}a^{-1}x} \cdots \ \tilde{b}_n|_{a_n^{-1}a^{-1}x} \cdots \  \tilde{b}_1|_{a_1^{-1}a^{-\ell+1}x} \cdots \ \tilde{b}_n|_{a_n^{-1}a^{-\ell+1}x}.
\end{align*}
If $\len_B(\gamma\Vert_x)$ is not smaller than the $B$-length of $w$, then elements in $\tilde{B}_{j+1}^{\alpha,\beta}$ and $A_{j+1}$ occur in an alternating way and every letter $\tilde{b}_i$ contributes one letter from $\tilde{B}_{j+1}^{\alpha,\beta}$.

Assume that $n \geq 2$ and assume that  $\len_B(\gamma\Vert_x) = \len_B(w)$.
Since $\tilde{b}_1$ contributes a $B$-letter, there is a $k$ such that $a_1^{-1}a^{-k}x = o$, i.e., $x \in \langle a \rangle a_1o$. 
In addition, at least $n-1$ letters in  $A_{j+1}$-letter occur in expression for the stabilized section.
Thus at least one of $\tilde{b}_1,\tilde{b}_2,\dots, \tilde{b}_{n}$ contributes a non-trivial element of $A_{j+1}$ and this means that $\alpha_j$ or $\beta_j$ lies in 
\[
 \{a_1^{-1},a_2^{-1},\dots,a_{n}^{-1}\}\langle a \rangle x \subseteq  \{a_1^{-1},a_2^{-1},\dots,a_{n}^{-1}\}\langle a \rangle a_1 o.
\]
We conclude that there are at most 
\[ n\ord(a)|Y_j'| + n\ord(a)|Y_j| \leq n m_j (|Y_j|+|Y'_j|)\]
 possibilities for the pair $(\alpha_j,\beta_j)$.

Assume now that $n = 1$, i.e., $w= \ ^{a_1} b_1 a$.
Using our assumption $\len(w) > (1,0)$, we see that $\tilde{b}_1$ needs to contribute an element from $A_{j+1}$ and so either $\alpha_j$ or $\beta_j$ is contained in
\[
a_1^{-1}\langle a \rangle x \subseteq a_1^{-1}\langle a \rangle a_1 o.
\]
Now this implies that there are at most $\ord(a)(|Y_j|+|Y'_j|) \leq m_j(|Y_j|+|Y'_j|)$ possibilities for $(\alpha_j,\beta_j)$.
 \end{proof}

 \begin{proof}[Proof of Proposition \ref{prop:shrinking}]
 Since $Q$ and $G$ are finitely generated, the groups $F_j = A_j*B$ are finitely generated and countable. 
 
Let $w \in F_j$ with $\len(w) > (1,0)$ be given. We define subsets $M(w)$ in $\mathcal{Y}\times \mathcal{Y}'$. If $\len_B(w) > 1$, then $M(w)$ consists of those pairs $(\alpha,\beta)$ such that for every $k \in \N$ there is a vertex $x$ of level $k$ in $\T_j$ such that
 $\len_B(f_j^{\alpha,\beta}(w))\Vert_x) = \len_B(w)$. If $\len_B(w) = 1$, we define 
 $M(w)$ to be the set of pairs  such that for every $k \in \N$ there is a vertex $x$ of level $k$ in $\T_j$ satisfying $\len_B(f_j^{\alpha,\beta}(w))\Vert_x) > (1,0)$.
 
 We note that the set of shrinking elements in $\mathcal{Y}\times \mathcal{Y}'$ is the complement of the union of all $M(w)$ over all $j$ and $w \in F_j$ with $\len(w) > (1,0)$.
 We claim that under our assumptions $M(w)$ is closed and nowhere dense. By Baire's category theorem a countable union of nowhere dense sets is nowhere dense and this proves the Proposition.
 
Write $M_k(w)$ for the set of all $(\alpha,\beta)\in \mathcal{Y}\times \mathcal{Y}'$ such that there is a vertex $x$ of level $k$ in $\T_j$ with $\len_B(f_j^{\alpha,\beta}(w))\Vert_x) = \len_B(w)$ (resp. with $\len(f_j^{\alpha,\beta}(w))\Vert_x) > (1,0)$ if $\len_B(w) = 1$).
Whether or not a pair $(\alpha,\beta)$ belongs to $M_k(w)$ depends only on the first $k$ entries of $\alpha$ and $\beta$. In particular, $M_k(w)$ is closed and open in $\mathcal{Y}\times \mathcal{Y}'$.
Since $M(w) = \bigcap_{k=1}^\infty M_k(w)$, we deduce that $M(w)$ is closed.
We claim that $M(w)$ is nowhere dense. Assume that $\len_B(w) > 1$.  Let $(\alpha,\beta) \in M(w)$ and assume for a contradiction that an open neighbourhood of $(\alpha,\beta)$ is contained in $M(w)$, i.e., there is a $\kappa \in \N$ such that every element $(\alpha',\beta')$  with $\alpha'_i = \alpha_i$ and $\beta'_i = \beta_i$ for all $i \leq \kappa+j$ lies in $M(w)$.
By assumption there is $\ell > \kappa$ such that
\[
|Y_{j+\ell}| |Y'_{j+\ell}| > \len_B(w)m_{j+\ell}(|Y_{j+\ell}|+|Y'_{j+\ell}|).
\]
Let $z = yx$ be a vertex on level $\ell+1$ of $\T_j$, where $y$ is a vertex of level $\ell$ and $x \in X_{j+\ell}$ satisfying $\len_B(f_j^{\alpha,\beta}(w)\Vert_{z}) = \len_B(w)$. Recall that $f_j^{\alpha,\beta}(w)\Vert_{z} = f_j^{\alpha,\beta}(w)\Vert_{y}\Vert_x$. 
Let $w' \in F_{j+\ell}$ be a representative of $f_j^{\alpha,\beta}(w)\Vert_{y}$ with $\len_B(w') = \len_B(w)$.
Then $\len(f_{j+\ell}^{\alpha,\beta}(w')\Vert_x) = \len_B(w) = \len_B(w')$ and hence $(\alpha_{j+\ell},\beta_{j+\ell})$ lies $Z_{j+\ell}(w')$, which by assumption is strictly smaller than $Y_{j+\ell}\times Y'_{j+\ell}$. This provides a contradiction and so $M(w)$ cannot contain any open set. 
In the case $\len_B(w) = 1$ essentially the same argument shows that $M(w)$ is nowhere dense.
\end{proof}
 
\begin{remark}
It seems possible to obtain a measure theoretic strengthening of Proposition~\ref{prop:shrinking}.
If we equip each $Y_j$ and $Y'_j$ with the uniform probability measure and equip $\mathcal{Y} \times \mathcal{Y}'$ with the product measure $\mu$, then under suitable assumptions we expect that $\mu$-almost all elements in $\mathcal{Y} \times \mathcal{Y}'$ are shrinking.
 \end{remark}

\section{Proof of the main theorem}
We first recall a simple lemma.
\begin{lemma}\label{lem:products}
A direct product of groups $G_1\times \dots \times G_k$ contains a non-abelian free group if and only if one of the groups $G_i$ contains a non-abelian free group.
\end{lemma}
\begin{proof}
This follows by induction and the observation that every pair of non-trivial normal subgroups in a non-abelian free group intersects non-trivially.
\end{proof}
Let $\Alt(n)$ denote the alternating group of degree $n$. For $k < n$ we identify $\Alt(k)$ with the subgroup of $\Alt(n)$ that permutes $\{1,2,\dots,k\}$.
\begin{lemma}\label{lem:altalt}
Let $F$ be a finite group of order $n$. There is an embedding of $F$ in $\Alt(2n+3)$ 
such that $\Alt(2n+3)$ is generated by the $F$-conjugates of $\Alt(5)$.
\end{lemma}
\begin{proof}
Let $F$ act freely on the sets $\{3\}\cup\{6,7,\dots,n+4\}$ and $\{4\} \cup \{n+5,n+6,\dots, 2n+3\}$. We take the diagonal action to ensure that $F$ acts on $\{1,\dots,2n+3\}$ by even permutations.
If we conjugate the $3$-cycle $(1,2,3)$ with all elements in $F$ we obtain all $3$-cycles of the form $(1,2,y)$ with $6 \leq y \leq n+4$. If we conjugate the $3$-cycle $(1,2,4)$ with all elements in $F$, we obtain all $3$-cycles of the form $(1,2,y)$ with $n+5\leq y$. It is well-known that the set of $3$-cycles $\{(1,2,y)\mid 3 \leq y \leq 2n+3\}$ generates $\Alt(2n+3)$. We deduce that the $F$-conjugates of $\Alt(5)$ generate $\Alt(2n+3)$.
 \end{proof}

Let $\mathfrak{C}$ be a class of groups that is closed under taking subgroups and finite direct products. Many prominent classes of groups meet these requirements: torsion groups, nilpotent groups, solvable groups, groups that don't contain a non-abelian free group, amenable groups, linear groups, groups satisfying a law $w$, etc.
We say that a group $\Gamma$ is \emph{locally $\mathfrak{C}$ up to powers}, if for every finite list $\gamma_1,\dots,\gamma_k \in \Gamma$, there are exponents $t_1,\dots,t_k \in \N$ such that $\langle \gamma_1^{t_1},\dots, \gamma_k^{t_k} \rangle$ is in $\mathfrak{C}$.
\begin{lemma}\label{lem:C-product}
If $G,H$ are locally $\mathfrak{C}$ up to powers, then $G \times H$ is locally $\mathfrak{C}$ up to powers.
\end{lemma}
\begin{proof}
Let $\gamma_1 = (g_1,h_1),\dots,\gamma_r = (g_r,h_r) \in G\times H$. There are exponents $s_1,\dots,s_r$ and $t_1,\dots,t_r$ such that $G_1=\langle g_1^{s_1}, \dots, g_r^{s_r}\rangle$ and $H_1 = \langle h_1^{t_1}, \dots, h_r^{t_r} \rangle$ are in $\mathfrak{C}$.
The subgroup $\langle \gamma_1^{s_1t_1},\gamma_2^{s_2t_2}, \dots,\gamma_r^{s_rt_r}\rangle$ is contained in $G_1 \times H_1$ and hence lies in $\mathfrak{C}$.
\end{proof}

\begin{theorem}\label{thm:C}
Let $G$ be a finitely generated residually finite group that is locally $\mathfrak{C}$ up to powers. Then $G$ embeds into a finitely generated perfect branch group $\Gamma$ that is locally $\mathfrak{C}$ up to powers.
\end{theorem}
\begin{proof}
We can assume that $G$ is infinite (if $G$ is finite, we may replace $G$ with $G\times G_2$ where $G_2$ is a finitely generated, infinite, amenable, residually finite, torsion group; for instance, Grigorchuk's group \cite{Grigorchuk80}). We set $Q = \Alt(5)$.
Let $N_1, N_2, \dots$ be a sequence of finite index normal subgroups of $G$ such that
\begin{equation}\label{eq:intersection}
 \bigcap_{i\geq j} N_i = \{1\}
\end{equation}
holds for all $j \geq 1$.

Let $n_i = |G/N_i|$ and let $A_i = \Alt(2n_i+3)$. We apply Lemma \ref{lem:altalt} and embed $G/N_i$ into $A_i$ such that the $G/N_i$ conjugates of $\Alt(5)$ generate $A_i$. Let $\sigma_i$ be a $3$-cycle in $A_i$. We define $X_i = A_i/\langle\sigma_i\rangle$ and pick $o = \langle \sigma_i \rangle$ as the special point.
Note that $G$, $Q$ and the $A_n$ satisfy the standing assumptions (A1)-(A3) of Section~\ref{sec:stabilized}.

Let $Y_i \subseteq X_i$ be the set of elements whose stabilizer in $A_i$ is \emph{not} $\langle \sigma_i \rangle$.
Let $Y_i'$ denote the set of elements different from $o$ whose stabilizer is $\langle \sigma_i \rangle$, i.e. $Y_i' = N_{A_i}(\langle \sigma_i \rangle)/\langle\sigma_i\rangle \setminus \{o\}$. 
The normalizer $N_{A_i}(\langle \sigma_i \rangle)$ is isomorphic to the group of even permutations in $\Sym(3) \times \Sym(2n_i)$; in particular
$|Y_i'| = \frac{6 \cdot (2n_i)!}{2 \cdot 3} -1 =  (2n_i)!-1$ and $|Y_i| = \frac{(2n_i+3)!}{3}-(2n_i!)$.

Let $g\colon \N \to \N$ be Landau's function, i.e., $g(n)$ is the maximal order of an element in the symmetric group $\Sym(n)$.
We verify that the conditions of Proposition \ref{prop:shrinking} are fulfilled.
To this end we will show that $\frac{|Y_i|\cdot|Y_i'|}{g(2n_i+3) (|Y_i|+|Y_i'|)}$ tends to infinity with $n_i$.
We simplify the expression
\begin{align*}
	\frac{|Y_i|\cdot|Y_i'|}{g(2n_i+3) (|Y_i|+|Y_i'|)} &\geq \frac{|Y_i|\cdot|Y_i'|}{g(2n_i+3) |X_i|}\\
	&= (1-\frac{3(2n_i)!}{(2n_i+3)!})\frac{(2n_i)!-1}{g(2n_i+3)}\\
	&\geq \frac{19}{20} \cdot \frac{(2n_i)!-1}{g(2n_i+3)}
\end{align*}
using that $\frac{3(2n)!}{(2n+3)!}= \frac{3}{(2n+3)(2n+2)(2n+1)}$ is monotonically decreasing with $n$ and attains the value $\frac{1}{20}$ for $n = 1$. It is a result of Landau that $\log(g(n)) \sim \sqrt{n\log(n)}$. However, for our purposes the elementary estimate $g(n) \leq \frac{n!}{2^{n-1}}$ is sufficient (induction!).
Then we obtain
\begin{align*}
	\frac{|Y_i|\cdot|Y_i'|}{g(2n_i+3) (|Y_i|+|Y_i'|)} &\geq \frac{19}{20} \cdot \frac{(2n_i)!-1}{g(2n_i+3)}\\
	&\geq \frac{19 \cdot 2^{2n_i+1}}{20 (2n_i+3)(2n_i+2)(2n_i+1)}
\end{align*}
Since $G$ is infinite, condition \eqref{eq:intersection} ensures that the $n_i$ and hence the left hand side is unbounded.

By Proposition \ref{prop:shrinking} the set of shrinking elements is dense in $\mathcal{Y}\times\mathcal{Y}'$.
In particular, we can find at least one shrinking pair $(\alpha,\beta)$. By the definition of $Y_i$, the elements $\alpha_j \in Y_j$ satisfy the stabilizer condition of Lemma~\ref{lem:branch-group-arg}. We conclude that
$\Gamma_1^{\alpha,\beta}$ is a branch group.  Using \eqref{eq:intersection} it follows that the subgroup $\tilde{G}_j^\beta$ of $\Gamma_j^{\alpha,\beta}$ is isomorphic to $G$.
We have to check that $\Gamma_1^{\alpha,\beta}$ is locally $\mathfrak{C}$ up to powers.
Let $\gamma_1,\dots, \gamma_r \in \Gamma_1^{\alpha,\beta}$. 
By Lemma \ref{lem:use-shrink} there is a level $k\in \N$ and there are exponents $m_1,\dots,m_r \in \N$ such that all sections of $\gamma_i^{m_i}$ on level $k$ lie in $\tilde{G}_{k+1}^\beta$ and $\gamma_i^{m_i} \in \St_{\Gamma_1^{\alpha,\beta}}(k)$.
Since $\tilde{G}_{k+1}^\beta \cong G$ the elements $\gamma_i^{m_i}$ generate a group that is isomorphic to a subgroup of $G^t$ where $t$ denotes the number of vertices on the $k$-th level. Since $G^t$ is locally $\mathfrak{C}$ up to powers (see \ref{lem:C-product}), we conclude that $\Gamma_1^{\alpha,\beta}$ satisfies $w$ up to powers.

\end{proof}

\begin{proof}[Proof of Theorem~\ref{thm:intro-main}]
Ad \eqref{it-torsion}: Let $\mathfrak{C}$ be the class of torsion groups. Being locally $\mathfrak{C}$ up to powers is the same as being torsion, so the result follows from Theorem \ref{thm:C}.

Ad \eqref{it-free}: Let $\mathfrak{C}$ be the class of group that don't contain a non-abelian free group. By Lemma~\ref{lem:products} the class $\mathfrak{C}$ is closed under subgroups and finite direct products.
We note that a group $H$ which is locally $\mathfrak{C}$ up to powers doesn't contain a non-abelian free group. 
In fact, if $\gamma_1,\gamma_2 \in H$ generate a non-abelian free group $\langle \gamma_1, \gamma_2 \rangle$, then $\langle \gamma_1^{t_1},\gamma_2^{t_s} \rangle$ is again a non-abelian free group.
Again the assertion follows immediately from Theorem \ref{thm:C}.

Ad \eqref{it-law}: Let $\mathfrak{C}$ be the variety of groups defined by the law $w$. Satisfying $w$ up to powers is the same as being locally $\mathfrak{C}$ up to powers and Theorem \ref{thm:C} proves the claim.

Ad \eqref{it-amenable}: Assume that $G$ is amenable. As before we can see that $\Gamma$ is locally amenable up to powers. To prove amenability we need a stronger result. There are various methods to prove amenability of groups acting on rooted trees.
Here we use Theorem 22 of \cite{JuschenkoNekrashevychdelaSalle16}, which in our case tells us that $\Gamma_1^{\alpha,\beta}$ is amenable if the germs of the action of $\Gamma_1^{\alpha,\beta}$ on $\T_1$ are amenable.
Since these can be easily seen to be either trivial or isomorphic to $Q \times G$, it follows that $\Gamma_1^{\alpha,\beta}$ is amenable if $G$ is amenable.
\end{proof}

\begin{proof}[Proof of Corollary~\ref{cor:amenable-action}]
Let $\gamma \in \Gamma_1^{\alpha,\beta}$ be an element in $A_1 \cup \tilde{G}^{\beta}_{[1]} \cup \tilde{Q}^\alpha_{[1]}$.
By construction each level of $\T_1$ contains at most $2$ vertices $u,v$ such that $g$ has non-trivial sections at $u$ and $v$.
Since $\Gamma_1^{\alpha,\beta}$ is generated by $A_1 \cup \tilde{G}^{\beta}_{[1]} \cup \tilde{Q}^\alpha_{[1]}$, it follows that for every $\gamma \in \Gamma_1^{\alpha,\beta}$ the number of non-trivial sections of $\gamma$ in each level is uniformly bounded.
In this case the main theorem of~\cite{GrigorchukNekrashevych07} tells us that the natural continuous action of $\Gamma^{\alpha,\beta}_1$ on the boundary of $\T_1$ restricts to amenable and faithful actions on the orbits of this action.
\end{proof}

\begin{proof}[Proof of Additio~\ref{additio:profinite}]
Let $G$ be as in Theorem \ref{thm:intro-main}.
Let $H$ be a profinite group and assume that $G \subseteq H$ is a dense subgroup. If $G$ is finite, then $H = G$ and there is nothing left to do. We assume that $G$ is infinite.

Let $\bar{N}_i \trianglelefteq H$ be a sequence of open normal subgroups with trivial intersection such that
 $N_i = G \cap \bar{N}_i$ meets the requirements of our construction.
By \cite[Corollary 4.4]{KionkeSchesler21} $\Gamma_1^{\alpha,\beta}$ has the congruence subgroup property. This means, that every finite index subgroup of $\Gamma_1^{\alpha,\beta}$ contains a level stabilizer. The $k$-th level stabilizer $\St_{\Gamma_1^{\alpha,\beta}}(k)$ intersects $\tilde{G}^\beta_{[1]}$  exactly in $\widetilde{L(k)}_1^\beta$ where $L(k) = \bigcap_{i \leq k} N_i \trianglelefteq G$. We deduce that the profinite topology of $\tilde{G}_{[1]}^\beta$ induced by the finite index subgroups of $\Gamma_1^{\alpha,\beta}$ agrees with the profinite topology inherited from the open subgroups of $H$. 
\end{proof}

\bibliographystyle{amsplain}
\bibliography{literatur}

\end{document}